\documentclass[11pt]{article}
\usepackage{fullpage}
\usepackage{amsfonts, amsthm, amssymb, amsmath}
\usepackage{helvet}
\usepackage{framed}
\usepackage{verbatim}
\usepackage{epsfig}
\usepackage[pdftex, pagebackref=true, colorlinks=true, urlcolor=blue, citecolor=blue, linkcolor=blue]{hyperref}

\newtheorem{thm}{Theorem}
\newtheorem{lem}[thm]{Lemma}
\newtheorem{cor}[thm]{Corollary}
\newtheorem{defn}[thm]{Definition}
\newtheorem{clm}[thm]{Claim}
\newtheorem{cons}[thm]{Construction}
\newtheorem{prop}[thm]{Proposition}

\newtheorem{conj}[thm]{Conjecture}
\newtheorem{obs}[thm]{Observation}

\newenvironment{theorem}{\begin{thm}\begin{rm}}%
{\end{rm}\end{thm}}
{\end{rm}\end{lem}}
{\end{rm}\end{cor}}
\newenvironment{definition}{\begin{defn}\begin{em}}%
{\end{em}\end{defn}}
{\end{rm}\end{clm}}
{\end{em}\end{cons}}
{\end{em}\end{prop}}
\newenvironment{conjecture}{\begin{conj}\begin{rm}}%
{\end{rm}\end{conj}}
{\end{rm}\end{obs}}

\newcommand{\secref}[1]{\hyperref[#1]{Section \ref{#1}}}
\newcommand{\thref}[1]{\hyperref[#1]{Theorem \ref{#1}}}
\newcommand{\defref}[1]{\hyperref[#1]{Definition \ref{#1}}}
\newcommand{\cororef}[1]{\hyperref[#1]{Corollary \ref{#1}}}
\newcommand{\propref}[1]{\hyperref[#1]{Proposition \ref{#1}}}
\newcommand{\remref}[1]{\hyperref[#1]{Remark \ref{#1}}}
\newcommand{\lemref}[1]{\hyperref[#1]{Lemma \ref{#1}}}
\newcommand{\clref}[1]{\hyperref[#1]{Claim \ref{#1}}}
\newcommand{\consref}[1]{\hyperref[#1]{Construction \ref{#1}}}
\newcommand{\figref}[1]{\hyperref[#1]{Figure \ref{#1}}}
\newcommand{\eqnref}[1]{\hyperref[#1]{Equation \ref{#1}}}
\newcommand{\subroutineref}[1]{\hyperref[#1]{Subroutine \ref{#1}}}
\newcommand{\apref}[1]{\hyperref[#1]{Appendix \ref{#1}}}
\newcommand{\conjref}[1]{\hyperref[#1]{Conjecture  \ref{#1}}}
\newcommand{\obsref}[1]{\hyperref[#1]{Observation \ref{#1}}}

\newcommand{\ds}{\displaystyle}

\newcommand{\tw}{treewidth}
\newcommand{\pw}{pathwidth}
\newcommand{\dtw}{directed treewidth}
\newcommand{\kw}{Kelly-width}
\newcommand{\dgw}{DAG-width}
\newcommand{\dpw}{directed pathwidth}

\newcommand{\dfstree}{depth-first search tree}

\title{Directed Width Parameters and Circumference of Digraphs}
\vspace{0.15in}
\author{
{Shiva Kintali} \\
\vspace{0.10in} \\
Department of Computer Science, \\
Princeton University, \\
Princeton, NJ 08540-5233. \\
{\small
\href{mailto:kintali@cs.princeton.edu}{\nolinkurl{kintali@cs.princeton.edu}}
}
}
\begin{document}
\maketitle
\begin{abstract}
We prove that the \dtw, \dgw\ and \kw\ of a digraph are bounded above by its circumference plus one. \\
%This generalizes a theorem of Birmele stating that the \tw\ of an undirected graph is at most its circumference minus one. \\

\end{abstract}

\section{Introduction}\label{sec:intro}

The {\em circumference} of an undirected graph (resp. digraph) $G$, denoted by ${\sf circ}(G)$, is the length of a longest simple undirected (resp. directed) cycle in $G$. The circumference of a DAG is defined to be one. The circumference of an undirected tree is defined to be two. Birmele \cite{birmele-circumference} proved that the \tw\ of an undirected graph $G$, denoted by ${\sf tw}(G)$, is at most its circumference minus one.

\begin{theorem}
(Birmele \cite{birmele-circumference}) For an undirected graph $G$, ${\sf tw}(G) \leq {\sf circ}(G) -1$.
\end{theorem}

Motivated by the success of \tw\ in algorithmic and structural graph theory, efforts have been made to generalize \tw\ to digraphs. Johnson et al. \cite{dtw-definition} introduced the first directed analogue of \tw\ called \dtw. Berwanger et al. \cite{dgw-definition1} and independently Obdrzalek \cite{dgw-definition2} introduced \dgw. Hunter and Kreutzer \cite{kellywidth-definition} introduced \kw. For a digraph $G$, let ${\sf dtw}(G)$, ${\sf dgw}(G)$ and ${\sf kw}(G)$ denote its \dtw, \dgw\ and \kw\ respectively. All these {\em directed} width measures are generalizations of undirected \tw\ i.e., for an {\em{undirected}} graph $G$, let $\overset\leftrightarrow{G}$ be the {\it digraph} obtained by replacing each edge $\{u,v\}$ of $G$ by two directed edges $(u,v)$ and $(v,u)$, then: 

\begin{itemize}
\item ${\sf dtw}(\overset\leftrightarrow{G}) = {\sf tw}(G)$ \cite[Theorem 2.1]{dtw-definition}
\item ${\sf dgw}(\overset\leftrightarrow{G}) = {\sf tw}(G)+1$ \cite[Proposition 5.2]{dgw-definition1}
\item ${\sf kw}(\overset\leftrightarrow{G}) = {\sf tw}(G)+1$ \cite{kellywidth-definition}
\end{itemize}

We prove that the \dtw, \dgw\ and \kw\ of a digraph are bounded above by its circumference plus one. Our proofs generalize Birmele's idea of constructing a tree decomposition using a \dfstree. For the \dtw\ we construct an arboreal decomposition from the \dfstree\ very naturally. The underlying arborescence is the \dfstree\ itself. For the \dgw\ and \kw\ we construct the underlying DAG using the \dfstree\ and some carefully chosen additional edges. Constructing the corresponding ``bags" requires some additional work to satisfy the strict guarding conditions of DAG-decompositions and Kelly-decompositions. Our main theorem is as follows:

\begin{theorem}\label{thm:main-theorem}
For a digraph $G$,
\begin{itemize}
\item ${\sf dtw}(G) \leq {\sf circ}(G) + 1$
\item ${\sf dgw}(G) \leq {\sf circ}(G) + 1$
\item ${\sf kw}(G) \leq {\sf circ}(G) + 1$
\end{itemize}
\end{theorem}

Birmele's theorem is tight as ${\sf tw}(K_n) = n-1$ and ${\sf circ}(K_n) = n$. Since ${\sf dtw}(\overset\leftrightarrow{K_n}) = n-1$, ${\sf dgw}(\overset\leftrightarrow{K_n}) = n$ and ${\sf kw}(\overset\leftrightarrow{K_n}) = n$, we conjecture that \thref{thm:main-theorem} can be improved with the following tight bounds:

\begin{conjecture}
For a digraph $G$,
\begin{itemize}
\item ${\sf dtw}(G) \leq {\sf circ}(G) - 1$
\item ${\sf dgw}(G) \leq {\sf circ}(G)$
\item ${\sf kw}(G) \leq {\sf circ}(G)$
\end{itemize}
\end{conjecture}

Birmele's theorem does not hold for \pw\ since complete binary trees have unbounded \pw. Nesetril and Ossona de Mendez \cite{sparsity-book} showed that the \pw\ of a 2-connected graph $G$ is at most $({\sf circ}(G)-2)^2$. Marshall and Wood \cite{pathwidth-circumference-improved} improved this bound to ${\lfloor{\sf circ}(G)/2\rfloor}$ $({\sf circ}(G)-1)$. Generalizing these results to \dpw, under a suitable {\em directed connectivity} assumption is an interesting open problem.

\subsection{Notation}\label{sec:notation}

We use standard graph theory notation and terminology (see \cite{diestel-textbook}). All digraphs are finite and simple (i.e. no self loops and no multiple arcs). For a digraph $G$, we write $V(G)$ for its vertex set and $E(G)$ for its arc set. For $S \subseteq V(G)$ we write $G[S]$ for the subdigraph induced by $S$, and $G \setminus S$ for the subdigraph induced by $V(G) - S$.

We use the term DAG when referring to directed acyclic graphs. A node is a {\it root} if it has no incoming arcs. The DAG $T$ is an {\it arborescence} if it has a unique root $r$ such that for every node $i \in V(T)$ there is a unique directed walk from $r$ to $i$. Note that every arborescence arises from an undirected tree by selecting a root and directing all edges away from the root.

Let $T$ be a DAG. For two {\em distinct} nodes $i$ and $j$ of $T$, we write $i \prec_T j$ if there is a directed walk in $T$ with first node $i$ and last node $j$. For convenience, we write $i \prec j$ whenever $T$ is clear from the context. For nodes $i$ and $j$ of $T$, we write $i \preceq j$ if either $i=j$ or $i \prec j$. For an arc $e=(i,j)$ and a node $k$ of $T$, we write $e \prec k$ if either $j=k$ or $j \prec k$. We write $e \sim i$ (and $e \sim j$) to mean that $e$ is incident with $i$ (and $j$ respectively). We define $T_{\succeq v} = T[\{ x\ |\ {x \succeq v} \}]$.

Let $\mathcal{W}=(W_i)_{i \in V(T)}$ be a family of finite sets called {\it node bags}, which associates each node $i$ of $T$ to a node bag $W_i$. We write $W_{\succeq i}$ to denote $\ds\bigcup_{j \succeq i} W_j$. For an arc $e$ of $T$, we write $W_{\succ e}$ to denote $\ds\bigcup_{j \succ e} W_j$. Let $\mathcal{A}=(A_e)_{e \in E(T)}$ be a family of finite sets called {\it arc bags}, which associates each arc $e$ of $T$ to an arc bag $A_e$. We write $A_{\sim i}$ to denote $\ds\bigcup_{e \sim i}A_e$.

\subsection{Guarding, $X$-normal and Directed unions}

Width measures like \dgw\ and \kw\ are based on the following notion of {\it guarding}:
\begin{definition}[Guarding]\label{definition:guarding}
Let $G$ be a digraph and $W, X \subseteq V(G)$. We say $X$ {\it guards} $W$ if $W \cap X = \emptyset$, and for all $(u,v) \in E(G)$, if $u \in W$ then $v \in W \cup X$.
\end{definition}

In other words, $X$ guards $W$ means that there is no directed path in $G \setminus X$ that starts from $W$ and leaves $W$. The notion of \dtw\ is based on a weaker condition:
\begin{definition}[$X$-normal]\label{definition:normal}
Let $G$ be a digraph and $W, X \subseteq V(G)$. We say $W$ is {\it $X$-normal} if $W \cap X = \emptyset$, and there is no directed path in $G \setminus X$ with first and last vertices in $W$ that uses a vertex of $G \setminus (W \cup X)$.
\end{definition}

In other words, $W$ is $X$-normal means that there is no directed path in $G \setminus X$ that starts from $W$, leaves $W$ and then returns to $W$. A digraph $D$ is a {\em directed union} of digraphs $D_1$ and $D_2$ if $D_1$ and $D_2$ are induced subgraphs of $D$, $V(D_1) \cup V(D_2) = V(D)$, and no edge of $D$ has head in $V(D_1)$ and tail in $V(D_2)$. The \dtw, \dgw\ and \kw\ are closed under directed unions (see \cite{dtw-definition, dgw-journal, recog-kw-two}). The following theorem is immediate.

\begin{theorem}
(\cite{dtw-definition, dgw-journal, recog-kw-two}) The \dtw\ (resp. \dgw, \kw) of a digraph $G$ is equal to the maximum \dtw\ (resp. \dgw, \kw) taken over the strongly-connected components of $G$.
\end{theorem}

Also, the circumference of a digraph $G$ is equal to the maximum circumference taken over the strongly-connected components of $G$. Hence, we may assume that all digraphs are strongly-connected in the rest of this paper.

\subsection{Depth-first search tree}\label{sec:dfs-tree}

Let $G$ be a strongly-connected digraph. Let $T$ be a \dfstree\ of $G$ starting at an arbitrary root $r \in V(G)$. The tree $T$ is an arborescence rooted at $r$. The edges of $G$ are classified into one of the four types : {\em tree edges}, {\em forward edges}, {\em back edges} and {\em cross edges} (see \cite{CLRS-algorithms-textbook}). For a vertex $v \in V(G)$, let $dfs(v)$ be the ``dfs number" of $v$ i.e., the time-stamp assigned to $v$ when $v$ is visited for the first time during the construction of $T$.

\section{Directed \tw\ and Circumference}

\begin{definition}[Arboreal decomposition and \dtw\cite{dtw-definition}]\label{definition:arboreal-decomposition}
An {\it arboreal decomposition} of a digraph $G$ is a triple $\mathcal{D} = (T, \mathcal{W}, \mathcal{A})$, where $T$ is an arborescence, and $\mathcal{W}=(W_i)_{i\in V(T)}$ is a family of subsets (node bags) of $V(G)$, and $\mathcal{A}=(A_e)_{e\in E(T)}$ is a family of subsets (arc bags) of $V(G)$, such that:
\begin{itemize}
\item { $\mathcal{W}$ is a partition of $V(G)$. \hfill{\rm{\sf (DTW-1)}}  }
\item { For each arc $e \in E(T)$, $W_{\succ e}$ is $A_e$-normal. \hfill{\rm{\sf (DTW-2)}} }
\end{itemize}
The width of an arboreal decomposition $\mathcal{D}=(T,\mathcal{W},\mathcal{A})$ is defined as $\max\{|W_i \cup A_{\sim i}|:i \in V(T)\} - 1$. The {\it \dtw} of $G$, denoted by ${\sf dtw}(G)$, is the minimum width over all possible arboreal decompositions of $G$.
\end{definition}

\begin{theorem}
For a digraph $G$, ${\sf dtw}(G) \leq {\sf circ}(G) + 1$.
\end{theorem}

\begin{proof}
Let $T$ be the \dfstree\ constructed in \secref{sec:dfs-tree}. Let $\mathcal{W}=(W_i)_{i\in V(T)}$ be a partition of $V(G)$ defined as $W_i = \{i\}$ for each $i \in V(T)$. For every edge $e = (r,v) \in E(T)$, we define $A_e = \{r\}$. For every edge $e = (u,v) \in E(T)$ such that $u \neq r$ we define $A_e$ as follows:

\begin{itemize}
\item{if there are no back edges from $W_{\succ e}$, we define $A_e = \{r\}$.}
\item{if there are back edges from $W_{\succ e}$, let $B$ be the set of all vertices $b \preceq u$ such that there is a back edge from some vertex in $W_{\succ e}$ to $b$. Let $b_0$ be the minimal element in $B$ with respect to $\preceq$. Let $A_e = \{r\} \cup \{x\ |\ b_0 \preceq x \preceq u\}$. Note that $|\{x\ |\ b_0 \preceq x \preceq u\}| \leq l-1$ and hence $|A_e| \leq l$.}
\end{itemize}

Let $\mathcal{A}=(A_e)_{e\in E(T)}$. We claim that $\mathcal{D}=(T,\mathcal{W},\mathcal{A})$ is an arboreal decomposition of $G$ of width at most $l+1$. By construction, $\mathcal{W}=(W_i)_{i\in V(T)}$ is a partition of $V(G)$ so $\mathcal{D}$ satisfies {\sf (DTW-1)}. To show that $\mathcal{D}$ satisfies {\sf (DTW-2)} we must show that for each arc $e \in E(T)$, $W_{\succ e}$ is $A_e$-normal. For every edge $e = (r,v) \in E(T)$, every directed path that leaves $W_{\succ e}$ and returns to $W_{\succ e}$ must go through the root $r$. Hence, $W_{\succ e}$ is $A_e$-normal. For every edge $e = (u,v) \in E(T)$ such that $u \neq r$ we consider the following cases:

\begin{itemize}
\item{if there are no back edges from $W_{\succ e}$, every directed path that leaves $W_{\succ e}$ and returns to $W_{\succ e}$ must go through the root $r$. Hence, $W_{\succ e}$ is $A_e$-normal.}
\item{if there are back edges from $W_{\succ e}$, every directed path that leaves $W_{\succ e}$ and returns to $W_{\succ e}$ must go through the root $r$ (or) go through a vertex in $\{x\ |\ b_0 \preceq x \preceq u\}$. Hence, $W_{\succ e}$ is $A_e$-normal.}
\end{itemize}

The size of each arc bag is at most $l$. Let $e_1 = (u,v), e_2 = (v,w) \in E(T)$. Let $B_1$ be the set of all vertices $b \preceq u$ such that there is a back edge from some vertex in $W_{\succ e_1}$ to $b$. Let $B_2$ be the set of all vertices $b' \preceq v$ such that there is a back edge from some vertex in $W_{\succ e_2}$ to $b'$. Note that $B_2 \subseteq B_1 \cup \{v\}$. Hence, for every $i \in V(T)$ the number of vertices in $A_{\sim i}$ is at most $l+1$ and the number of vertices in $W_i \cup A_{\sim i}$ is at most $l+2$. Hence, the width of $\mathcal{D}=(T,\mathcal{W},\mathcal{A})$ is at most $l+1$.
\end{proof}

\section{\dgw\ and Circumference}

\begin{definition}[DAG-decomposition and \dgw\ \cite{dgw-definition1,dgw-definition2,dgw-journal}]\label{definition:DAG-decomposition}
A {\it DAG decomposition} of a digraph $G$ is a pair $\mathcal{D} = (T,\mathcal{X})$ where $T$ is a DAG, and $\mathcal{X} = (X_i)_{i \in V(T)}$ is a family of subsets (node bags) of $V(G)$, such that:
\begin{itemize}
\item $\bigcup_{i \in V(T)} X_i = V(G)$. \hfill{\rm{\sf (DGW-1)}}
\item For all nodes $i,j,k \in V(T)$, if $i \preceq j \preceq k$, then $X_i \cap X_k \subseteq X_j$. \hfill{\rm{\sf (DGW-2)}}
\item For all arcs $(i,j) \in E(T)$, $X_i \cap X_j$ guards $X_{\succeq j} \setminus X_i$. For any root $r \in V(T)$,\\
$X_{\succeq r}$ is guarded by $\emptyset$. \hfill{\rm{\sf (DGW-3)}}
\end{itemize}
The width of a DAG-decomposition $\mathcal{D}=(T,\mathcal{X})$ is defined as $\max\{|X_i|:i \in V(T)\}$. The {\it \dgw} of $G$, denoted by ${\sf dgw}(G)$, is the minimum width over all possible DAG-decompositions of $G$.
\end{definition}

\begin{theorem}\label{thm:dgw-circumference}
For a digraph $G$, ${\sf dgw}(G) \leq {\sf circ}(G) + 1$.
\end{theorem}

\begin{proof}
Let $T$ be the \dfstree\ constructed in \secref{sec:dfs-tree}. We construct a DAG $\tilde{T}$ by adding more edges to $T$. For a vertex $v \in V(T)$ let $S_v =\{u\ |\ dfs(u) < dfs(v)\ \mbox{and}\ u \nprec_{T} v\}$. Add new edges from $v$ to every vertex in $S_v$. We do this for every $v \in V(T)$. The graph $\tilde{T}$ obtained in this way is a DAG.

We now define the set of node bags $\mathcal{X}=(X_v)_{v\in V(\tilde{T})}$. Let $X_r = \{r\}$. For every vertex $v \neq r$, we define $X_v$ as follows:

\begin{itemize}
\item{if there are no back edges from $T_{\succeq v}$, we define $X_v = \{r\}$.}
\item{if there are back edges from $T_{\succeq v}$, let $B$ be the set of all vertices $b \preceq_T v$ such that there is a back edge from some vertex in $T_{\succeq v}$ to $b$. Let $b_0$ be the minimal element in $B$ with respect to $\preceq_T$. Let $X_v = \{r\} \cup \{x\ |\ b_0 \preceq x \preceq v\}$. Note that $|\{x\ |\ b_0 \preceq x \preceq v\}| \leq l$ and hence $|X_v| \leq l+1$.}
\end{itemize}

The size of each node bag is at most $l+1$. We claim that $\mathcal{D}=(\tilde{T},\mathcal{X})$ is a DAG decomposition of $G$. Note that $V(G) = V(\tilde{T})$ and $v \in X_v$ for every vertex $v\in V(\tilde{T})$. Hence, {\sf (DGW-1)} is satisfied.

Consider two vertices $i \neq j$ such that $i \in X_j$. There exist $b \preceq i$ and $a \succeq j$ such that $(a,b) \in E(G)$ is a back edge. Every vertex $k$ such that $i \preceq k \preceq j$ satisfies $b \preceq k \preceq j$, and hence by our construction $k \in X_j$. So, {\sf (DGW-2)} is satisfied.

All the out-going edges from $X_{\succeq j} \setminus X_i$ are either back edges (or) edges going through the root $r$. All the heads of the back edges from $X_{\succeq j} \setminus X_i$ are in $X_i \cap X_j$. Also, $r \in X_v$ for every $v\in V(\tilde{T})$. Hence, $X_i \cap X_j$ guards $X_{\succeq j} \setminus X_i$ and {\sf (DGW-3)} is satisfied.
\end{proof}

\section{\kw\ and Circumference}

\indent Kelly-decomposition and \kw\ were introduced by Hunter and Kreutzer \cite{kellywidth-definition}.
\begin{definition}[Kelly-decomposition and \kw\ \cite{kellywidth-definition}]\label{definition:Kelly-decomposition}
A {\it Kelly-decomposition} of a digraph $G$ is a triple $\mathcal{D} = (T, \mathcal{W}, \mathcal{X})$ where $T$ is a DAG, and $\mathcal{W}=(W_i)_{i\in V(T)}$ and $\mathcal{X}=(X_i)_{i\in V(T)}$ are families of subsets (node bags) of $V(G)$, such that:
\begin{itemize}
\item $\mathcal{W}$ is a partition of $V(G)$. \hfill{\rm{\sf (KW-1)}}
\item For all nodes $i \in V(T), X_i$ guards $W_{\succeq i}$. \hfill{\rm{\sf (KW-2)}}
\item For each node $i \in V(T)$, the children of $i$ can be enumerated as $j_1, ... , j_s$ so that for each $j_q$, $X_{j_q} \subseteq W_i \cup X_i \cup \bigcup_{p<q} W_{\succeq j_p}$. Also, the roots of $T$ can be enumerated as $r_1, r_2, ...$ such that for each root $r_q$, $X_{r_q} \subseteq \bigcup_{p<q} W_{\succeq r_p}$. \hfill{\rm{\sf (KW-3)}}
\end{itemize}
The width of a Kelly-decomposition $\mathcal{D}=(T,\mathcal{W},\mathcal{X})$ is defined as $\max\{|W_i \cup X_i| : i \in V(T)\}$. The {\it \kw} of $G$, denoted by ${\sf kw}(G)$, is the minimum width over all possible Kelly-decompositions of $G$.
\end{definition}

\begin{theorem}
For a digraph $G$, ${\sf kw}(G) \leq {\sf circ}(G) + 1$.
\end{theorem}

\begin{proof}
Let $\tilde{T}$ be the DAG constructed in the proof of \thref{thm:dgw-circumference}. Let $\mathcal{W}=(W_i)_{i\in V(T)}$ be a partition of $V(G)$ defined as $W_i = \{i\}$ for each $i \in V(T)$. We now define the set of node bags $\mathcal{X}=(X_v)_{v\in V(\tilde{T})}$. Let $X_r = \emptyset$. For every vertex $v \neq r$, we define $X_v$ as follows:

\begin{itemize}
\item{if there are no back edges from $T_{\succeq v}$, we define $X_v = \{r\}$.}
\item{if there are back edges from $T_{\succeq v}$, let $B$ be the set of all vertices $b \preceq_T v$ such that there is a back edge from some vertex in $T_{\succeq v}$ to $b$. Let $b_0$ be the minimal element in $B$ with respect to $\preceq_T$. Let $X_v = \{r\} \cup \{x\ |\ b_0 \preceq x \preceq v\} \setminus v$. Note that $|\{x\ |\ b_0 \preceq x \preceq v\} \setminus v| \leq l-1$ and hence $|X_v| \leq l$. We call $b_0$ the ``hook" of $v$ and denote it by $hook(v)$.}
\end{itemize}

The size of each node bag is at most $l$, so the size of each $|W_i \cup X_i|$ is at most $l+1$. We claim that $\mathcal{D}=(\tilde{T},\mathcal{W},\mathcal{X})$ is a Kelly decomposition of $G$. By construction, $\mathcal{W}=(W_i)_{i\in V(T)}$ is a partition of $V(G)$ so $\mathcal{D}$ satisfies {\sf (KW-1)}.

All the out-going edges from $W_{\succeq i}$ are either back edges (or) edges going through the root $r$. All the heads of the back edges from $W_{\succeq i}$ are in $X_i$. Also, $r \in X_v$ for every $v\in V(\tilde{T})$. Hence, $X_i$ guards $W_{\succeq i}$ and {\sf (KW-2)} is satisfied.

Recall the definition of ``hook". For a vertex $v \in V(T)$, if there are no back edges from $T_{\succeq v}$, we define $hook(v) = v$. For a node $i \in V(T)$, we enumerate the children of $i$ as $j_1, \dots , j_s$ such that $hook(j_1) \succeq hook(j_2) \succeq \dots \succeq hook(j_s)$. With this ordering, {\sf (KW-3)} is satisfied.
\end{proof}

\bibliographystyle{alpha}
\bibliography{../bib-twbook,../bib-kintali}

\end{document}